\documentclass[11pt,a4paper]{amsart}
\usepackage{amssymb,amsfonts,amsmath,mathrsfs}
\usepackage{comment}
\usepackage[left=2cm,right=2cm,top=2cm,height=250mm]{geometry}
\usepackage{color}

\newtheorem{proposition}{Proposition}[section]
\newtheorem{theorem}[proposition]{Theorem}
\newtheorem{corollary}[proposition]{Corollary}
\newtheorem{lemma}[proposition]{Lemma}

\theoremstyle{definition}
\newtheorem{definition}[proposition]{Definition}

\newtheorem{remark}[proposition]{Remark}

\numberwithin{equation}{section}

\def\R{\Bbb R}
\def\N{\Bbb N}
\def\Dx{\Delta_x}
\def\Nx{\nabla_x}
\def\Dt{\partial_t}
\def\({\left(}
\def\){\right)}
\def\eb{\varepsilon}

\def\Cal{\mathcal}

\def\E{\mathcal{E}}
\def\A{\mathcal{A}}
\def\Bbb{\mathbb}

\begin{document}
\title[Attractor for hyperbolic Cahn-Hilliard-Oono equation]{Finite Dimensionality of the attractor for the hyperbolic Cahn-Hilliard-Oono Equation in $\R^3$}
\author[A. Savostianov and S. Zelik] {Anton Savostianov and Sergey Zelik}

\begin{abstract} In this paper, we continue the study of the hyperbolic relaxation of the Cahn-Hilliard-Oono equation with the sub-quintic non-linearity in the whole space $\R^3$ started in our previous paper and verify that under the natural assumptions on the non-linearity and the external force, the fractal dimension of the associated global attractor in the natural energy space is finite.
\end{abstract}
\thanks{This work is partially supported by  the grant  14-41-00044 of RSF and the grant 14-01-00346 of RFBR}


\address{University of Surrey, Department of Mathematics, Guildford, GU2 7XH, United Kingdom.}

\email{a.savostianov@surrey.ac.uk}
\email{s.zelik@surrey.ac.uk}
\maketitle
\section{Introduction}
The classical Cahn-Hilliard (CH) equation
\begin{equation}\label{0.CH}
\Dt u+\Dx(\Dx u-f(u)+g)=0,\ \ u\big|_{t=0}=u_0,
\end{equation}
where $u=u(t,x)$ is the so-called order parameter, $f(u)$ is a given nonlinearity, $g$ is a given external force, and $\Dx$ is the Laplacian with respect to the variable $x\in\Omega\subset\R^N$, is central for the theory of phase transitions and material sciences, see \cite{CH,CMZ,Ell,No1} and references therein. It also worth to note that  the sole equation \eqref{0.CH} is not sufficient for the accurate description of the whole variety of physical phenomena arising in this theory, so a number of various modifications of this equation has been introduced, see \cite{BaE,Bo,D,DKS,EK,EK1,EKZ,ElS,GJ,GGMP1DM,GPS,MZ3,MZ2,MZ1,Miro,Oop,PZ} and references therein.
\par
One of the interesting from both mathematical and physical points of view modifications of the CH equation is the following hyperbolic relaxation of the CH equation or hyperbolic CH equation:
\begin{equation}\label{0.hyp}
\eb\Dt^2 u+\Dt u+\Dx(\Dx u-f(u)+g)=0,\ \ \eb>0,
\end{equation}
which has been introduced by P. Galenko and coauthors (see \cite{GJ,GL1,GL2,GL3}) in order to treat in a more accurate way the non-equilibrium effects in spinodal decomposition. In a fact, the inertial term $\eb\Dt^2 u$ changes drastically the type of the equation (from parabolic to hyperbolic) and the analytical properties of its solutions. Moreover, the nonlinearity $\Dx(f(u))$ becomes "critical" even if the equation is considered in the class of smooth solutions and "supercritical" if the estimate $u(t)\in L^\infty(\Omega)$ is not available. By this reason, despite a big current interest (see e.g., \cite{GGMP1D,GGMP3D,GGMP1DM,GSZ,GSZ1,GSZ3,GPS}), the global well-posedness of equation \eqref{0.hyp} in a bounded domain $\Omega\subset\R^N$ is established only in the case $N=1$ and $N=2$. Thus, in the most interesting 3D case, only global existence of weak energy solutions is known even in the case of bounded nonlinearities $f(u)$.
\par
The case when the underlying domain is a whole space $\Omega=\R^3$ is surprisingly simpler due to the recent work \cite{SZ.CHO}, where the global well-posedness of problem \eqref{0.hyp} is established in a slightly stronger (than energy ones) class of solutions (the so-called Strichartz solutions) and the nonlinearity $f$ of the sub-quintic growth rate (see below for more details). This result is obtained combining the  technique of energy estimates with the classical Strichartz estimates for the linear Schr\"odinger equation in $\R^3$, see \cite{tao} and references therein. Unfortunately these estimates do not work or unknown  for general bounded domains, see \cite{sogge1}, so the extension of these results to bounded domains remains an open problem.
\par
We also remind that, in contrast to the case of bounded domains, the dissipation is naturally lost in the long-wave limit
$u_0=u_0(\mu x)$, $\mu\to0$, in the case when $\Omega=\R^3$, so problem \eqref{0.hyp} (as well as the initial problem \eqref{0.CH}) becomes non-dissipative at least in the usual sense and does not possess a compact global attractor (see \cite{DKS} for the partial dissipativity results for the case of viscous CH equation in $\R^N$ and \cite{CM,PZ} for the non-dissipative bounds for solutions, see also \cite{Bo,EKZ} for the case of pipe-like domains where the Poincar\'e inequality restores the dissipativity). By this reason, it seems natural to take into the account one more physically relevant modification of the CH equation and consider the hyperbolic relaxation of the  Cahn-Hilliard-Oono (CHO) equation:
\begin{equation}
\label{eqCHO}
\begin{cases}
\Dt^2u+\Dt u+\Dx(\Dx u-f(u)+g)+\alpha u=0,\ x\in\R^3,\\
u\big|_{t=0}=u_0,\ \ \Dt u\big|_{t=0}=u_1,
\end{cases}
\end{equation}
where the extra term $\alpha u$, $\alpha>0$, describes the so-called long-range interactions, see \cite{Miro,Oop} for more details. From the mathematical point of view, this extra term does not change the type of the equation and the analytic properties of solutions, but on the other hand, it removes the aforementioned long wave instability and restores the dissipation mechanism which allows us to use the machinery of global attractors to study the long time behaviour of its solutions.
\par
The hyperbolic CHO equation and its global attractor are the main objects to study in the present notes. We recall that this equation possesses at least formally the energy equality in the form
\begin{equation}\label{0.en}
\frac d{dt}\(\|\Dt u\|^2_{\dot H^{-1}}+\|u\|^2_{\dot H^1}+\alpha\|u\|^2_{\dot H^{-1}}+2(F(u),1)-2(g,u)\)=-2\|\Dt u\|^2_{\dot H^{-1}},
\end{equation}
where $\dot H^s:=D((-\Dx)^s)$ are the homogeneous Sobolev spaces in $\R^3$ (with $\|u\|^2_{\dot H^s}:=((-\Dx)^su,u)$, see \cite{triebel} for more details concerning these spaces),
$F(u)=\int_0^uf(v)\,dv$ is the potential of the nonlinearity $f$ and $(u,v)$ stands for the standard inner product in $L^2(\R^3)$. Thus, the natural energy phase space for problem \eqref{eqCHO} is the following one:
\begin{equation}
\label{e.sp}
\E=[\dot{H}^1\cap \dot{H}^{-1}]\times\dot{H}^{-1},\quad \|\xi_u\|^2_\E=\|\Dt u\|^2_{\dot{H}^1}+\alpha\|u\|^2_{\dot{H}^{-1}}+\|u\|^2_{\dot H^{-1}},
\end{equation}
where $\xi_u(t):=(u(t),\Dt u(t))$.
Following \cite{SZ.CHO}, we also assume that $g\in \dot H^1$ and the nonlinear function $f\in C^2(\R^3)$ satisfies the following dissipativity and growth assumptions
\begin{equation}\label{f123}
\begin{cases}
1. & f(u)u\geq 0,\\
2. & F(u)\leq Lf(u)u+K|u|^2,\\
3. & |f''(u)|\leq C(1+|u|^{3-\kappa}),
\end{cases}
\end{equation}
for some strictly positive constants $L$, $K$ and $\kappa\in(0,3]$. The natural definition of {\it energy} solutions associated with the identity \eqref{0.en} would be $\xi_u\in C(0,T;\mathcal E)$ (or $\xi_u\in L^\infty(0,T;\E)$ (and satisfy \eqref{eqCHO} is the sense of distributions). However, as already mentioned above, the uniqueness theorem is not known for such solutions, so following again \cite{SZ.CHO}, we will use slightly stronger class of solutions.

\begin{definition}\label{def S.sol}
A function $u=u(t,x)$ is a {\it Strichartz} solution of problem \eqref{eqCHO} if, for any $T>0$,
\begin{equation}
\xi_u\in C(0,T;\Cal E),\ \ u\in L^4(0,T;C_b(\R^3))
\end{equation}
and equation \eqref{eqCHO} is satisfied as an equality in $\dot H^{-1}+\dot H^{-3}$.
\end{definition}
The extra regularity $u\in L^\infty(0,T;C_b(\R^3))$ which is based on Strichartz estimates for the Schr\"odinger equation is crucial for proving the uniqueness of the Strichartz solutions proved in \cite{SZ.CHO}. Namely, the following result is proved in \cite{SZ.CHO}.

\begin{theorem}
\label{th.wp}
Let the nonlinearity $f$ satisfy assumptions \eqref{f123} and  the external force $g\in\dot H^1(\R^3)$. Then, for every initial data $\xi_0\in\E$ there exists and unique a global Strichartz solution $u$ of problem \eqref{eqCHO} such that $\xi_u(0)=\xi_0$. Furthermore, the following dissipative estimate holds:
\begin{equation}\label{2.en-dis}
\|\xi_u(t)\|_{\Cal E}+\|u\|_{L^4([t,t+1];C_b(\R^3))}\le Q(\|\xi_u(0)\|_{\Cal E})e^{-\beta t}+Q(\|g\|_{\dot H^1}), \ \ t\ge0,
\end{equation}
where the positive constant $\beta$ and monotone function $Q$ are independent of $t\ge0$ and the solution $u$.
\end{theorem}

Thus, under the above assumptions, problem \eqref{eqCHO} generates a dissipative semigroup $S(t)$, $t\ge0$, in the energy phase space via the following expression:
\begin{equation}\label{St}
S(t):\E\to\E,\quad S(t)\xi_0=\xi_u(t),
\end{equation}
where $u$ is a Strichartz solution to equation \eqref{eqCHO} with initial data $\xi_0$. As verified in \cite{SZ.CHO}, this semigroup possesses a smooth global attractor (see Definition \ref{Def2.attr}) in the phase space $\E$. To be more precise, the following result is proved there.

\begin{theorem}\label{th.attr}
Let assumptions of Theorem \ref{th.wp} be satisfied then the semigroup $S(t)$ defined by \eqref{St} possesses a compact global attractor $\Cal A$ in in the energy phase space $\E$ which is  bounded in more regular space $\E_2:=[\dot H^3\cap\dot H^{-1}]\times[\dot H^1\cap\dot H^{-1}]$.
\end{theorem}

The aim of the present notes which can be considered as a continuation of the study  initiated in \cite{SZ.CHO}   is to establish the finiteness of fractal dimension of the attractor $\Cal A$  constructed in Theorem \ref{th.attr}. Thus, the main result of the notes is the following theorem.

\begin{theorem}\label{Th0.main} Let the assumptions of Theorem \ref{th.wp} hold. Then the global attractor $\Cal A\subset\E$ of the solution semigroup $S(t)$ associated with the hyperbolic CHO equation \eqref{eqCHO} has the finite fractal (box counting) dimension:
\begin{equation}\label{0.df}
\dim_f(\mathcal A,\E)<\infty.
\end{equation}
\end{theorem}

To prove this theorem we utilize the classical volume contraction method, see \cite{TemamDS} (see also \cite{Ab2,BV1,EZ01} and references therein for the applications of this method to the case of unbounded domains). However, in contrast to the cases considered there, in our case it looks difficult/impossible to estimate the volume contraction factor using the Liouville formula in the initial metric of the space $\E$. So, similarly to the case of damped driven Schr\"odinger equation considered in \cite{G.88}, we have to use the Liouville formula in $\mathcal E$ with the properly chosen time-dependent metric which becomes more complicated (in comparison with \cite{G.88}) since we need to overcome extra difficulties related with the fact that the underlying domain $\Omega=\R^3$ is unbounded, see Section \ref{s2} for more details.
\par
The paper is organized as follows.
\par
The Liouville formula for the expansion factors for $d$-dimensional volumes is reminded in Section~\ref{s1}. We pay  a special attention to the case where the metric in the underlying space is time-dependent which is crucial for  proving our main result.
\par
The volume contraction theorem is stated in Section \ref{s2}. We also remind some necessary definitions there and verify that the solution semigroup associated with the hyperbolic CHO equation is uniformly quasidifferentiable on the attractor.
\par
Finally, the proof of the main Theorem \ref{Th0.main} is given in Section \ref{s3}.

\section{Volume contraction factors, traces and Liouville's formula in the spaces with time-dependent metric}\label{s1}

For the convenience of the reader, we briefly recall in this section the key facts from the multi-linear algebra which will be used in the next section for the proof of the main result, see e.g., \cite{G.88,TemamDS} for more detailed exposition. We start with reminding the construction of  the $d$th exterior power of a Hilbert space $\E$.
\begin{definition}\label{Def1.ext}
Let $\E$ be a separable  Hilbert space and let $\varphi_1,\cdots,\varphi_d\in\E$. A wedge product $\varphi_1\wedge\cdots\wedge\varphi_n$ is a $d$-linear anti-symmetric form on $\E$ defined by the following expression:
\begin{equation}\label{1.wedge}
\varphi_1\wedge\cdots\wedge\varphi_d(\psi_1,\cdots,\psi_d):=\det\((\varphi_i,\psi_j)_{i,j=1}^d\),\ \ \psi_1,\cdots,\psi_d\in\E.
\end{equation}
A $d$-linear form on $\E$ which is a wedge product of  $d$ vectors of $\E$ is called decomposable.
Let us denote by $\widetilde \Lambda^d\E$ the space of $d$-linear antisymmetric forms which can be presented as {\it finite} linear combination of decomposable functionals. For two decomposable forms $\varphi_1\wedge\cdots\wedge\varphi_d$ and $\psi_1\wedge\cdots\wedge\psi_d$, their inner product is defined as follows:
\begin{equation}\label{1.product}
(\varphi_1\wedge\cdots\wedge\varphi_d,\psi_1\wedge\cdots\wedge\psi_d):=\det\((\varphi_i,\psi_j)_{i,j=1}^d\)
\end{equation}
and being extended by linearity to $\widetilde\Lambda^d\E$ it defines an inner product on the space  $\widetilde\Lambda^d\E$, see \cite{TemamDS} for the details. Finally, the completion of $\widetilde\Lambda^d\E$ with respect to this norm is called $d$th exterior power of the space $\E$ and is denoted by $\Lambda^d\E$.
\end{definition}
\begin{remark}\label{Rem1.vol} We remind that the representation of a decomposable form $\varphi_1\wedge\cdots\varphi_d$ as a wedge product of $d$ vectors of $\E$ is not unique. Moreover, using the Gram orthogonalization procedure, it is easy to show that there exists an {\it orthogonal} system of vectors $\bar\varphi_1,\cdots,\bar\varphi_d\in\E$ such that
\begin{equation}\label{1.ort}
\varphi_1\wedge\cdots\wedge\varphi_d=\bar\varphi_1\wedge\cdots\wedge\bar\varphi_d,\ \ \|\bar\varphi_i\|_{\E}\leq\|\varphi_i\|_{\E},\ \ i=1,\cdots,d.
\end{equation}
Remind also that the norm $\|\varphi_1\wedge\cdots\wedge\varphi_d\|_{\Lambda^d\E}$ can be interpreted as the $d$-dimensional volume of the parallelepiped generated by the vectors $\varphi_1,\cdots,\varphi_d\in\E$ and, in particular,
\begin{equation}\label{1.cont}
\|\varphi_1\wedge\cdots\wedge\varphi_d\|_{\Lambda^d\E}\le \|\varphi_1\|_{\E}\cdots\|\varphi_d\|_{\E}.
\end{equation}
Furthermore, if $\{e_i\}_{i=1}^\infty$ is an orthonormal basis in $\E$ then any $d$-linear antisymmetric form $\xi$ on $\E$ has the form
$$
\xi=\sum_{i_1<i_2<\cdots<i_d} a_{i_1,\cdots,i_d}e_{i_1}\wedge\cdots\wedge e_{i_d},\ \ a_{i_1,\cdots,i_d}=\xi(e_1,\cdots,e_d)\in\R
$$
and
$$
\|\xi\|_{\Lambda^d\E}^2=\sum_{i_1<i_2<\cdots<i_d}a_{i_1,\cdots,i_d}^2.
$$
Then, as not difficult to see, any $\xi\in\Lambda^d\E$ is a $d$-linear continuous form on $\E$, so $\Lambda^d\E$ is a subset of $d$-linear continuous antisymmetric forms on $\E$. This subset is proper if $d>1$ and $\dim\E=\infty$. For instance, if $d=2$, the space of $2$-linear antisymmetric forms on $\E$ is naturally identified (via $\xi_A(\psi_1,\psi_2)\mapsto (A\psi_1,\psi_2)$) with the space of linear continuous antisymmetric operators on $\E$ and $\Lambda^2\E$ will be the space of antisymmetric Hilbert-Schmidt operators.
\end{remark}
We are now ready to define the $d$th exterior power of a linear continuous operator $L\in\mathcal L(\E,\E)$ which controls the change of $d$-dimensional volumes under the action of this operator.
\begin{definition} Let $L\in\Cal L(\E,\E)$ be a linear continuous operator on $\E$. The linear operator $\Lambda^dL$ acts on the space $\Lambda^d\E$ by the following expression:
\begin{equation}\label{1.Lext}
(\Lambda^dL)\xi(\psi_1,\cdots,\psi_d):=\xi(L^*\psi_1,\cdots,L^*\psi_d),\ \ \xi\in\Lambda^d\E,\ \ \psi_1,\cdots,\psi_d\in\E,
\end{equation}
where $L^*\in\mathcal L(\E,\E)$ is the adjoint operator to $L$. In particular, if $\xi=\varphi_1\wedge\cdots\wedge\varphi_d$ is
decomposable then
\begin{equation}
\Lambda^d L(\varphi_1\wedge\cdots\wedge\varphi_d)=(L\varphi_1)\wedge\cdots\wedge(L\varphi_d).
\end{equation}
Note also that as follows from \eqref{1.Lext}
\begin{equation}\label{1.mult}
\Lambda^d(L_1\circ L_2)=\Lambda^dL_1\circ\Lambda^d L_2
\end{equation}
for any two linear operators $L_1,L_2\in\mathcal L(\E,\E)$.
\end{definition}
The following result is proved, e.g., in \cite{TemamDS}.
\begin{proposition} Let $L\in\mathcal L(\E,\E)$. Then, $\Lambda^dL\in\mathcal L(\Lambda^d\E,\Lambda^d\E)$ and the following formula holds:
\begin{equation}\label{1.eqnorm}
\|\Lambda^dL\|_{\mathcal L(\Lambda^d\E,\Lambda^d\E)}=\omega_d(L):=\sup_{\varphi_1\wedge\cdots\wedge\varphi_d\ne0}
\frac{\|(L\varphi_1)\wedge\cdots\wedge(L\varphi_d)\|_{\Lambda^d\E}}{\|\varphi_1\wedge\cdots\wedge\varphi_d\|_{\Lambda^d\E}}.
\end{equation}
\end{proposition}
\begin{remark}\label{Rem1.vol1} Interpreting the norm of the wedge product as the volume of the corresponding parallelepiped, we may
read \eqref{1.eqnorm} as
\begin{equation}
\omega_d(L)=\sup_{\Pi\subset\E}\frac{\operatorname{vol}_d(L\Pi)}{\operatorname{vol}_d(\Pi)},
\end{equation}
where the supremum is taken over all non-degenerate $d$-dimensional parallelepipeds in $\E$. Thus, geometrically $\omega_d(L)$ is the maximal expanding factor for $d$-dimensional volumes under the action of the operator $L$. Mention also the equivalent definitions of the volume contraction factor
\begin{multline}\label{1.vol-eq}
\omega_d(L)=\sup\bigg\{\|(L\varphi_1)\wedge\cdots\wedge(L\varphi_d)\|_{\Lambda^d\E}\,: \varphi_1,\cdots,\varphi_d\in\E,\ \|\varphi_1\|_{\E}=\cdots=\|\varphi_d\|_{\E}=1\bigg\}=\\=\sup\bigg\{\|(L\varphi_1)\wedge\cdots\wedge(L\varphi_d)\|_{\Lambda^d\E}\,: \varphi_1,\cdots,\varphi_d\in\E,\ (\varphi_i,\varphi_j)=\delta_{ij}\bigg\}.
\end{multline}
Indeed, the equivalence can be easily verified using the Gram orthogonalization procedure, see \cite{TemamDS}.
\end{remark}
The next simple corollary gives useful estimates for the volume contraction factor $\omega_d(L)$.
\begin{corollary} Let $L\in\mathcal L(\E,\E)$. Then
\begin{equation}\label{1.vol-norm}
\omega_d(L)\le\|L\|^d_{\mathcal L(\E,\E)}.
\end{equation}
Moreover, if $L_1,L_2\in\mathcal L(\E,\E)$ then
\begin{equation}\label{1.vol-alg}
\omega_d(L_1L_2)\le\omega_d(L_1)\omega_d(L_2).
\end{equation}
\end{corollary}
Indeed, estimate \eqref{1.vol-norm} follows from \eqref{1.vol-eq} and \eqref{1.cont} and estimate \eqref{1.vol-alg} is an immediate corollary of the identities \eqref{1.mult} and \eqref{1.eqnorm}.
\par
At the next step, we introduce one more extension of the operator $L\in\mathcal L(\E,\E)$ to the exterior power $\Lambda^d\E$ which is responsible for the "trace part" of the Liouville formula.
\begin{definition}\label{Def1.trace} Let $L\in\mathcal L(\E,\E)$. For any $\xi\in\Lambda^d\E$, we define the $d$-linear antisymmetric functional $L_d\xi$ as follows:
\begin{equation}\label{1.trL}
(L_d\xi)(\psi_1,\cdots,\psi_d):=\xi(L^*\psi_1,\psi_2,\cdots,\psi_d)+\xi(\psi_1,L^*\psi_2,\cdots,\psi_d)+
\cdots+\xi(\psi_1,\cdots,\psi_{d-1},L^*\psi_d).
\end{equation}
In particular, if $\xi=\varphi_1\wedge\cdots\wedge\varphi_d$ is decomposable then
\begin{equation}\label{1.trLd}
L_d(\varphi_1\wedge\cdots\wedge\varphi_d)=(L\varphi_1)\wedge\varphi_2\wedge\cdots\wedge\varphi_d+\cdots+
\varphi_1\wedge\cdots\wedge\varphi_{d-1}\wedge(L\varphi_d).
\end{equation}
It is not difficult to show that $L_d\in\Cal L(\Lambda^d\E,\Lambda^d\E)$ and
\begin{equation}
\|L_d\|_{\mathcal L(\Lambda^d\E,\Lambda^d\E)}\le d\|L\|_{\mathcal L(\E,\E)}.
\end{equation}
However, we will not use this estimate in the sequel, so we leave its proof to the reader (see \cite{TemamDS} for its proof in the self-adjoint case).
\end{definition}
The next formula plays the crucial role in the derivation of the Liouville formula.
\begin{proposition} Let $L\in\mathcal L(\E,\E)$ and $\varphi_1,\cdots,\varphi_d\in\E$. Then
\begin{equation}\label{1.tr-form}
(L_d(\varphi_1\wedge\cdots\wedge\varphi_d),\varphi_1\wedge\cdots\wedge\varphi_d)_{\Lambda^d\E}=\operatorname{Tr}(Q\circ L\circ Q)\|\varphi_1\wedge\cdots\wedge\varphi_d\|^2_{\Lambda^d\E},
\end{equation}
where $Q=Q(\varphi_1,\cdots,\varphi_d)$ is the orthoprojector to the $d$-dimensional subspace in $\E$ spanned by the vectors $\varphi_1,\cdots,\varphi_d$ and $\operatorname{Tr}(Q\circ L\circ Q)$ is a usual trace of the $d$-dimensional operator (matrix) $Q\circ L\circ Q$ which can be computed as follows:
\begin{equation}\label{1.tr-fin}
\operatorname{Tr}(Q\circ L\circ Q)=\sum_{i=1}^d(L\psi_i,\psi_i),
\end{equation}
where $\{\psi_1,\cdots,\psi_d\}$ is any orthonormal system in $Q\E$.
\end{proposition}
Indeed, \eqref{1.tr-form} is obvious if $\varphi_1,\cdots,\varphi_d\in\E$ are orthogonal and the general case can be reduced to this particular one using \eqref{1.ort}, see \cite{TemamDS} for the details.
\par
We are now ready to state the key Liouville formula first for the case of time independent metric. To this end, we assume that we are given the following linear evolution equation in $\E$:
\begin{equation}\label{1.lin-eq}
\frac d{dt}\varphi(t)=L(t)\varphi(t),\ \ \varphi\big|_{t=0}=\varphi_0
\end{equation}
for some $L\in L^\infty(0,T;\mathcal L(\E,\E))$. Then, the following result holds.
\begin{proposition}\label{Prop1.Li-const} Let $\varphi_1(t),\cdots,\varphi_d(t)$ be the solutions of problem \eqref{1.lin-eq}. Then the following identity holds:
\begin{equation}\label{1.liu}
\frac12\frac d{dt}\|\varphi_1(t)\wedge\cdots\wedge\varphi_d(t)\|^2_{\Lambda^d\E}=\operatorname{Tr}(Q(t)\circ L(t)\circ Q(t))\|\varphi_1(t)\wedge\cdots\wedge\varphi_d(t)\|^2_{\Lambda^d\E},
\end{equation}
where $Q(t)$ is the orthoprojector to the $d$-dimensional space spanned by the vectors $\varphi_1(t),\cdots,\varphi_d(t)\in\E$.
\end{proposition}
\begin{proof} Indeed, multiplying equation \eqref{1.lin-eq} by $\varphi(t)$ in $\E$, we have
\begin{equation}\label{1.energy}
\frac12\frac d{dt}\|\varphi(t)\|^2_{\E}=(L(t)\varphi(t),\varphi(t))_{\E}=(L^{sym}(t)\varphi(t),\varphi(t))_{\E}
\end{equation}
and, therefore, due to the parallelogram law,
\begin{equation}\label{1.en1}
\frac12\frac d{dt}(\varphi_i(t),\varphi_j(t))_{\E}=(L^{sym}(t)\varphi_i(t),\varphi_j(t))_{\E},\ \ i,j=1,\cdots,d,
\end{equation}
where $L^{sym}(t)=\frac12(L(t)+L^*(t))$ is the symmetric part of the operator $L$. Differentiating now the determinant in the LHS of \eqref{1.liu} and using \eqref{1.en1} and \eqref{1.tr-form}, we have
\begin{multline}
\frac12\frac d{dt}\|\varphi_1(t)\wedge\cdots\wedge\varphi_d(t)\|^2_{\Lambda^d\E}=
(L^{sym}_d(t)(\varphi_1(t)\wedge\cdots\wedge\varphi_d(t)),\varphi_1(t)\wedge\cdots\wedge\varphi_d(t))_{\Lambda^d\E}=\\=
\operatorname{Tr}(Q(t)\circ L^{sym}(t)\circ Q(t))\|\varphi_1(t)\wedge\cdots\wedge\varphi_d(t)\|_{\Lambda^d\E}^2.
\end{multline}
Since, obviously,
$$
\operatorname{Tr}(Q(t)\circ L^{sym}(t)\circ Q(t))=\operatorname{Tr}(Q(t)\circ L(t)\circ Q(t)),
$$
then the proposition is proved.
\end{proof}
The trace on the RHS of the Liouville formula still depends on the vectors $\varphi_i(t)$ which are usually not known explicitly, so for its estimating it is convenient to introduce one more object.

\begin{definition}\label{Def1.trace1} Let $L\in\mathcal L(\E,\E)$. Then its $d$-dimensional trace is defined as the following number
\begin{equation}\label{1.dtrace}
\operatorname{Tr}_d(L):=\sup\bigg\{\sum_{i=1}^d(L\psi_i,\psi_i)\,: \ \psi_i\in\E,\ (\psi_i,\psi_j)=\delta_{ij}\bigg\}.
\end{equation}
Then, obviously,
\begin{equation}\label{1.good}
\operatorname{Tr}(Q\circ L\circ Q)\le \operatorname{Tr}_d(L).
\end{equation}
Moreover,
\par
1) if $L_1,L_2\in\mathcal L(\E,\E)$ then
\begin{equation}\label{1.tr-sum}
\operatorname{Tr}_d(L_1+L_2)\le\operatorname{Tr}_d(L_1)+\operatorname{Tr}_d(L_2).
\end{equation}
\par 2) If $L_1,L_2\in\mathcal L(\E,\E)$. Then
\begin{equation}
(L_1\varphi,\varphi)\le (L_2\varphi,\varphi)\ \ \forall\varphi\in\E\ \ \Rightarrow \ \ \operatorname{Tr}_d(L_1)\le\operatorname{Tr}_d(L_2).
\end{equation}
\end{definition}
\begin{corollary}\label{Co1.vol} Let $U(t,0):\E\to\E$ be the solution operator of equation \eqref{1.lin-eq} defined via the expression $U(t,0)\varphi(0)=\varphi(t)$. Then the volume expansion factor $\omega_d(U(t,0))$ possesses the following estimate:
\begin{equation}\label{1.vol-exp}
\omega_d(U(t,0))\le e^{\int_0^t\operatorname{Tr}_d(L(s))\,ds}.
\end{equation}
\end{corollary}
Indeed, integrating \eqref{1.liu} in time and using \eqref{1.good}, we have
$$
\|\varphi_1(t)\wedge\cdots\wedge\varphi_d(t)\|_{\Lambda^d\E}\le e^{\int^t_0\operatorname{Tr}_d(L(s))\,ds}\|\varphi_1(0)\wedge\cdots\wedge\varphi_d(0)\|_{\Lambda^d\E}
$$
and estimate \eqref{1.vol-exp} is now an immediate corollary of \eqref{1.eqnorm}.
\par
The next proposition is very useful for estimating the $d$-dimensional traces
\begin{proposition}\label{Prop1.min-max} Let $L\in\mathcal L(\E,\E)$ be self-adjoint and let
\begin{equation}\label{1.min-max}
\mu_k(L):=\inf_{F\subset\E,\,\dim F=k-1}\sup_{\varphi\in F^\perp,\,\varphi\ne0}\frac{(L\varphi,\varphi)_\E}{\|\varphi\|^2_{\E}},
\end{equation}
where the infinum is taken over all $(k-1)$-dimensional planes in $\E$ and $F^\perp$ stands for the orthogonal complement in $\E$. Then sequence $\mu_k(L)$ is monotone decreasing and, consequently, the limit
\begin{equation}\label{1.inf}
\mu_{\infty}(L)=\lim_{k\to\infty}\mu_k(L)
\end{equation}
exists. This limit coincides with the upper bound of the continuous spectrum of the operator $L$. Moreover,
\par
1) Any $\mu_k(L)>\mu_\infty(L)$ is an eigenvalue of the operator $L$ and, in particular, $\mu_\infty(L)=0$ if the operator $L$ is compact.
\par
2) The following formula for the $d$-dimensional traces hold:
\begin{equation}\label{1.tr-mu}
 \operatorname{Tr}_d(L)=\sum_{k=1}^d\mu_k(L).
\end{equation}
\end{proposition}
The proof of this proposition is based on the min-max principle, see \cite{TemamDS} for more details.

\begin{remark}\label{Rem1.unbounded} For simplicity, we consider above only the case where the operator $L(t)$ in equation \eqref{1.lin-eq} is bounded although it is usually {\it unbounded} in applications. However, the Liouville formula \eqref{1.liu} is actually finite-dimensional and we only need the solutions of \eqref{1.lin-eq} to be well-defined and satisfy the
 energy identity \eqref{1.energy} and, for the validity of Proposition \ref{Prop1.min-max}, we need the operator $L$ to be bounded from above. Note also that in our application to the case of hyperbolic CHO the energy equality will be automatically satisfied and the symmetric part $L^{sym}(t)$ will be a bounded operator.
\end{remark}

Our next task is to extend the Liouville formula to the case of time dependent metrics. To this end, we  assume that we are given a family $\|\cdot\|_{\E(t)}$ of time-dependent Hilbert norms in $\E$ such that
\begin{equation}\label{1.equiv}
c^{-1}\|\varphi\|_{\E}^2\le \|\varphi\|_{\E(t)}^2\le c\|\varphi\|_{\E}^2,\ \ t\in\R,\ \ c>0,
\end{equation}
where the constant $c$ is independent of $t$. Moreover, we assume that the solutions of equation \eqref{1.lin-eq} satisfy the energy equality of the form
\begin{equation}\label{1.en-time}
\frac12\frac {d}{dt}\|\varphi(t)\|^2_{\E(t)}=(M(t)\varphi(t),\varphi(t))_{\E}
\end{equation}
for some operators $M(t)$. Then, the following result holds.

\begin{proposition}\label{Prop1.lio-time} Let the solutions of equation \eqref{1.lin-eq} be well-posed and satisfy the energy identity \eqref{1.en-time}. Then, the following analogue of formula \eqref{1.liu} holds:
\begin{equation}\label{1.lio-time}
\frac12 \frac d{dt}\|\varphi_1(t)\wedge\cdots\wedge\varphi_d(t)\|^2_{\Lambda^d\E(t)}=\operatorname{Tr}(Q(t)\circ M_{\E(t)}(t)\circ Q(t))\|\varphi_1(t)\wedge\cdots\wedge\varphi_d(t)\|^2_{\Lambda^d\E(t)},
\end{equation}
where $Q(t)$ is the orthoprojector in $\E(t)$ to the space spanned by the vectors $\varphi_1(t),\cdots,\varphi_d(t)$ and $M_{\E(t)}(t)$ are such that
\begin{equation}\label{1.qforms}
(M(t)\varphi,\varphi)_{\E}=(M_{\E(t)}(t)\varphi,\varphi)_{\E(t)}
\end{equation}
(which exist due to the Riesz representation theorem).
\end{proposition}
Indeed, the proof of this statement repeats word by word the proof of Proposition \ref{Prop1.Li-const}, see also \cite{G.88}.
\par
The next two corollaries connect the volume contraction factors and traces in the spaces $\E$ and $\E(t)$.
\begin{corollary}\label{Cor1.det} Let the Hilbert norms $\|\cdot\|_{\E(t)}$ satisfy \eqref{1.equiv} and $L\in\mathcal L(\E,\E)$. Then
\begin{equation}\label{1.vol-eq1}
c^{-d}\omega_d(L,\E)\le\omega_d(L,\E(t))\le c^d\omega_d(L,\E),
\end{equation}
where $\omega_d(L,\E)$ and $\omega_d(L,\E(t))$ are volume expanding factors of $L$ in the spaces $\E$ and $\E(t)$ respectively.
\end{corollary}
\begin{proof}
Indeed, by Riesz representation theorem, there exist positive self-adjoint operators $U(t)=V(t)^\frac{1}{2}$ such that
\begin{equation}\label{1.U}
\|\varphi\|^2_{\E(t)}=(V(t)\varphi,\varphi)_\E=\|U(t)\varphi\|^2_{\E}.
\end{equation}
Moreover, estimate \eqref{1.equiv} give that
$$
\|U(t)\|_{\mathcal L(\E,\,\E)}\le c^{1/2},\ \ \|U(t)^{-1}\|_{\mathcal L(\E,\,\E)}\le c^{1/2}.
$$
Moreover, as not difficult to show,
$$
\omega_d(L,\E(t))=\omega_d(U(t)LU(t)^{-1},\E)
$$
and, thanks to  \eqref{1.vol-norm} and \eqref{1.vol-alg}
$$
\omega_d(L,\E(t))\le\omega_d(U(t),\E)\omega_d(L,\E)\omega_d(U(t)^{-1},\E)\le \|U(t)\|^d_{\mathcal L(\E,\,\E)}\|U(t)^{-1}\|^d_{\mathcal L(\E,\,\E)}\omega_d(L,\E)\le c^d\omega_d(L,\E).
$$
The opposite inequality can be proved analogously and the corollary is proved.
\end{proof}
\begin{corollary}\label{Cor1.tr} Let the Hilbert norms $\|\cdot\|_{\E(t)}$ satisfy \eqref{1.equiv} and  the operators $M(t)$ and $M_{\E(t)}(t)$ be such that \eqref{1.qforms} is satisfied. Then
\begin{equation}\label{1.tr-time}
\operatorname{Tr}_d(M_{\E(t)}(t),\E(t))\le c\operatorname{Tr}_d(M(t),\E)
\end{equation}
if the quadratic form $(M(t)\varphi,\varphi)$ is positive definite. The constant $c$ on the RHS of \eqref{1.tr-time} should be replaced by $c^{-1}$ if this form is negative definite.
\end{corollary}
\begin{proof} Without loss of generality we may assume that $M(t)$ and $M_{\E(t)}(t)$ are self-adjoint in $\E$ and $\E(t)$ respectively. Assume also that the quadratic form is positive (non-negative). The case of negative forms can be considered analogously. According to Proposition \ref{Prop1.min-max} it is enough to compare the corresponding eigenvalues. Using \eqref{1.equiv}, we have
\begin{multline}
\mu_k(M_{\E(t)}(t),\E(t))=\inf_{dim F=k-1}\sup_{\varphi\in F^\perp,\varphi\neq 0}\frac{(M_{\E(t)}(t)\varphi,\varphi)_{\E(t)}}{\|\varphi\|^2_{\E(t)}}=\\=\inf_{dim F=k-1}\sup_{\varphi\in F^\perp,\varphi\neq 0}\frac{(M(t)\varphi,\varphi)_{\E}}{\|\varphi\|^2_{\E(t)}}\le c\inf_{dim F=k-1}\sup_{\varphi\in F^\perp,\varphi\neq 0}\frac{(M(t)\varphi,\varphi)_{\E}}{\|\varphi\|^2_{\E}}
=c\mu_k(M(t),\E)
\end{multline}
and formula \eqref{1.tr-mu} finishes the derivation of \eqref{1.tr-time}. Thus, the corollary is proved.
\end{proof}
We are now ready to state the main result of the section which will be used for the proof of the finite-dimensionality of the attractor for the hyperbolic CHO equation.
\begin{theorem}\label{Th1.main} Let equation \eqref{1.lin-eq} be well-posed in $\E$ and its solutions $\varphi(t)$ possess the energy identity \eqref{1.en-time} where the Hilbert norms $\|\cdot\|_{\E(t)}$ satisfy \eqref{1.equiv} and the operators $M(t)$ can be estimated from above by  a sum
\begin{equation}\label{1.split}
(M(t)\varphi,\varphi)_{\E}\le(\mathcal C(t)\varphi,\varphi)_{\E}+(K(t)\varphi,\varphi)_{\E}, \ \ \varphi\in\E,
\end{equation}
where $\mathcal C(t)$ are negatively definite:
\begin{equation}\label{1.neg}
(\mathcal C(t)\varphi,\varphi)_{\E}\le -\alpha\|\varphi\|^2_{\E},\  \varphi\in\E
\end{equation}
with the constant $\alpha>0$ which is independent on $t$ and operators $K(t)$ are positive (non-negative) definite and possess the estimate
\begin{equation}\label{1.comp}
(K(t)\varphi,\varphi)_\E\le (K\varphi,\varphi)_{\E},
\end{equation}
where the operator $K\in\mathcal L(\E,\E)$ is compact. Then the volume expanding factor $\omega_d(U(t,0))$ of the solution operator $U(t,0)$ of problem \eqref{1.lin-eq} in $\E$ possesses the following estimate:
\begin{equation}\label{1.small}
\omega_d(U(t,0),\E)\le e^{d\ln c+(c C_K-\frac\alpha {2c} d)t},
\end{equation}
where the constant $C_K$ depends only on the operator $K$. In particular, if $d\in\Bbb N$ is chosen in such way that
\begin{equation}\label{1.dim}
c C_K-\frac\alpha {2c} d<0
\end{equation}
then
\begin{equation}\label{1.fin}
\omega_d(U(t,0),\E)\le \frac12,\ \ t\ge t_0,
\end{equation}
where $t_0$ depends only on $c$, $\alpha$, $C_K$ and $d$.
\end{theorem}
\begin{proof} According to the Liouville formula \eqref{1.lio-time} analogously to \eqref{1.vol-exp}, we have
\begin{equation}
\omega_d(U(t,0),\E(t))\le e^{\int_0^t\operatorname{Tr}_d(M_{\E(s)}(s),\,\E(s))\,ds}.
\end{equation}
Furthermore, thanks to \eqref{1.vol-eq1} and \eqref{1.tr-sum},
\begin{equation}\label{1.nice}
\omega_d(U(t,0),\E)\le c^d\omega_d(U(t,0),\E(t))\le c^d e^{\int_0^t\operatorname{Tr}_d(\mathcal C(s)_{\E(s)},\E(s))+\operatorname{Tr}_d(K(s)_{\E(s)},\E(s))\,ds}.
\end{equation}
Since $\mathcal C(t)$ is negative and $K(t)$ is positive, Corollary \ref{Cor1.tr} gives
\begin{equation}
\operatorname{Tr}_d(\mathcal C(t)_{\E(t)},\mathcal E(t))\le c^{-1}\operatorname{Tr}_d(\mathcal C(t),\E)\le -c^{-1}\alpha d
\end{equation}
and
\begin{equation}
\operatorname{Tr}_d(K(t)_{\E(t)},\mathcal E(t))\le c\operatorname{Tr}_d(K(t),\mathcal E)\le c\operatorname{Tr}_d(K,\E).
\end{equation}
Without loss of generality we may assume that the operator $K$ is self-adjoint. Since it  is compact by the assumptions of the theorem, then $\mu_\infty(K)=0$ and according to \eqref{1.tr-mu}, there exists a constant $C_K$ depending only on $K$ such that
$$
\operatorname{Tr}_d(K,\E)\le C_K+\frac\alpha{2c^2} d.
$$
Inserting the obtained estimate into the RHS of \eqref{1.nice}, we end up with the desired estimate \eqref{1.small}. Estimate \eqref{1.fin} is an immediate corollary of \eqref{1.nice} and the theorem is proved.
\end{proof}

\section{Box counting dimension and volume contraction theorem}\label{s2}

In this section, we state the so-called volume contraction theorem which is one of the main  technical tools for estimating the dimension of the attractor, see \cite{bkBV}, \cite{TemamDS}, and start to verify its assumptions for the case of hyperbolic CHO equation. We begin with reminding the key definitions.

\begin{definition}\label{DEf2.frac} Let $\mathcal A$ be a compact set in a metric space $\Cal E$. By Hausdorff criterium, for every $\eb>0$, $\mathcal A$ can be covered by finitely-many balls of radius $\eb$ in $\Cal E$. Let $N_\eb(\mathcal A,\Cal E)$ be the minimal number of such balls which is enough to cover $\Cal E$. Then, the fractal (box-counting) dimension of $\mathcal A$ is defined as follows:
\begin{equation}\label{5.frac}
\dim_f(\mathcal A,\Cal E):=\limsup_{\eb\to0}\frac{\log N_\eb(\mathcal A,\Cal E)}{\log\frac1\eb}.
\end{equation}
It is worth mentioning that in the case when $\mathcal A$ is regular enough, e.g., when it is a  Lipschitz manifold, the fractal dimension coincides with the usual dimension of the manifold. However, for  irregular sets it can easily be non-integer. For instance, the dimension of the standard ternary Cantor set in $[0,1]$ is $\frac{\ln2}{\ln3}$. We mention also that this dimension is always finite if $\E$ is finite dimensional, but a priori it can be infinite in the case of infinite-dimensional spaces $\E$.
\end{definition}

\begin{definition}
\label{def.qd}
A map $S:\Cal A\to\Cal A$, where $\Cal A$ is a compact subset of a Banach space $\E$, is called uniform quasidifferentiable on $\A$ if for any $\xi\in\Cal A$ there exists a linear operator $S'(\xi)\in\mathcal L(\E,\E)$ ( the quasidifferential) such that for any $\xi_1,\ \xi_2\in\Cal A$
\begin{equation}\label{2.dif}
\|S(\xi_2)-S(\xi_1)-S'(\xi_1)(\xi_2-\xi_1)\|_\E=o(\|\xi_1-\xi_2\|_\E),
\end{equation}
holds uniformly with respect to $\xi_1,\ \xi_2\in\Cal A$ and,  in addition,
\begin{equation}
\label{2.Scont}
S'(\xi)\in C(\Cal A,\Cal L(\E)).
\end{equation}
\end{definition}
We remark that the difference between quasidifferential and Frechet derivative is in the fact that for quasidifferential we consider increments only in those directions $\xi_2-\xi_1$ where $\xi_1,\xi_2\in \Cal A\subset\E$ whereas for the Frechet derivative one should consider all possible directions in $\E$. In particular this may lead to non-uniqueness of operator $S'(\xi)$. However this essentially relaxes assumptions on $S$ and makes this property easier to verify, especially when extra smoothness of $\Cal A$ is known, which is usually the case in the attractors theory .

The main abstract theorem of volume contraction method can be formulated as follows
\begin{theorem}\label{Th2.AVolContr}
Let $\Cal A$ be a compact subset of a Hilbert space $\E$ which is invariant with respect to map $S$, that is $S\Cal A=\Cal A$. Suppose that $S$ is quasidifferentiable on $\Cal A$. Suppose also that $S'(\xi)$ contracts all $d$-dimensional volumes uniformly with respect to $\xi\in\Cal A$, that is
\begin{equation}\label{2.cont}
\omega_d(\Cal A, S):=\sup_{\xi\in\Cal A} \omega_d(S'(\xi),\E)<1,
\end{equation}
 Then the fractal dimension of $\Cal A$ in the space $\E$ is finite and the following estimate holds:
\begin{equation*}
\dim_f(\Cal A,\E)\leq d.
\end{equation*}
\end{theorem}

The proof of this theorem can be found in \cite{TemamDS} for the case of Hausdorff dimension and in \cite{chil} for the case of fractal dimension.

Thus, we need to apply Theorem \ref{Th2.AVolContr} to the global attractor $\Cal A$ given by Theorem \ref{th.attr}. For the convenience of the reader, we remind also the definition of a global attractor, see e.g., \cite{bkBV} for more details.

\begin{definition}\label{Def2.attr} A set $\mathcal A\subset\mathcal E$ is a global attractor of a semigroup $S(t)$, $t\ge0$, acting in a metric space $\E$ if the following conditions are satisfied:
\par
1) The set $\mathcal A$ is compact in $\E$;
\par
2) It is strictly invariant, i.e., $S(t)\mathcal A=\mathcal A$ for all $t\ge0$;
\par
3) It attracts  the images of bounded sets of $\E$ as time tends to infinity, i.e., for any bounded set $B\subset\E$ and any neighbourhood $\mathcal O(\mathcal A)$ of the set $\mathcal A$ in $\E$, there is time $T=T(B,\mathcal O)$ such that
$$
S(t)B\subset\mathcal O(\mathcal A)
$$
for all $t\ge T$.
\end{definition}
In our situation the space $\E$ is the energy space defined by \eqref{e.sp} and the solution
semigroup $S(t)$ is defined by \eqref{St}. It is important that, due to Theorem \ref{th.attr}, the attractor $\mathcal A$ is bounded in $\E_2$ and, in particular, due to the embedding theorems,
\begin{equation}\label{2.c1}
\|u\|_{C^{1+\delta}(\R^3)}\le C=C_{\mathcal A},
\end{equation}
where $\delta<1/2$ and the constant $C$ is independent of $\xi_u\in\mathcal A$. Let us mention also that, due to the invariance of the attractor, it is generated by all bounded trajectories of \eqref{eqCHO} defined for all $t\in\R$:
\begin{equation}
\mathcal A=\mathcal K\big|_{t=0},
\end{equation}
where $\mathcal K\subset C_b(\R, \E)$ is the set of all complete bounded trajectories of \eqref{eqCHO}, see \cite{bkBV} for more details.

As usual, to estimate the dimension of the attractor $\mathcal A$, we will apply  Theorem \ref{Th2.AVolContr} with $S=S(T)$ where $T>0$ is a sufficiently large time. To this end, we first need to know that this map is quasidifferentiable on the attractor.
  As expected,  the quasidifferential of $S(t)$ can be found using equation in variations
\begin{theorem}\label{Th2.var}
Let assumptions of Theorem \ref{th.wp} hold. Then the solution operator $S(t)$ associated with problem \eqref{eqCHO} (see \eqref{St}) is uniformly quasidifferentiable on the attractor $\Cal A$  and its quasidifferential $S'(t,\xi_0)$ at point $\xi_0\in \Cal A$ can be found as $S'(t,\xi_0)\hat\xi:=\xi_w(t)$, where $w(t)$ solves the equation of variations
\begin{equation}\label{eq.var}
\begin{cases}
\Dt^2w+\Dt w+\alpha w+\Dx(\Dx w-f'(u(t))w)=0,\quad x\in \R^3,\\
\xi_w|_{t=0}=(w_0,w'_0):=\hat\xi\in \E,
\end{cases}
\end{equation}
where $u(t)=S(t)\xi_0$ is a Strichartz solution of equation \eqref{eqCHO} with the initial data $\xi_u(0)=\xi_0\in\mathcal A$.
\end{theorem}
\begin{proof} Since the assertion of the theorem is standard, we give below only the sketch of its proof leaving the details to the reader. First we need to establish the well-posedness of the equation of variations \eqref{eq.var}. This is done in the following lemma.
\begin{lemma}\label{Lem2.var} Under the above assumptions equation \eqref{eq.var} possesses a unique solution $\xi_w\in C(0,T;\E)$ for any $\hat\xi\in\E$ and any $\xi_0\in\mathcal A$. Moreover, the following estimate holds:
\begin{equation}\label{2.var-est}
\|\xi_w(t)\|_{\E}\le Ce^{Kt}\|\xi_w(0)\|_{\E},\ \ t\ge0,
\end{equation}
where the constants $C$ and $K$ are independent of $\hat\xi\in\E$ and $\xi_0\in\mathcal A$.
\end{lemma}
\begin{proof}[Proof of the lemma] We restrict ourselves to formal derivation of estimate \eqref{2.var-est} which can be justified exactly as in \cite{SZ.CHO}. To this end, we multiply equation \eqref{eq.var} by $\Dt(-\Dx)^{-1}w$ and integrate over $x\in\R^3$. Then after the standard transformations, we get
\begin{multline}\label{2.var-est1}
\frac12\frac d{dt}\|\xi_w(t)\|^2_{\E}+\|\Dt w\|^2_{\dot H^{-1}}=\\=-(f'(u(t)w,\Dt w)\le \|\Nx(f'(u(t)w)\|_{L^2}\|\Dt w\|_{\dot H^{-1}}\le \|\xi_w(t)\|^2_{\E}+\|\Nx(f'(u(t)w)\|_{L^2}^2.
\end{multline}
Using now estimate \eqref{2.c1} together with the embedding $\dot H^1\cap\dot H^{-1}\subset H^1$ (here and below $H^s=H^s(\R^3)$ stands for the usual non-homogeneous Sobolev spaces), we estimate the last term on the RHS as follows:
\begin{multline}\label{2.good}
\|\Nx(f'(u(t)w)\|_{L^2}^2=\|f'(u)\Nx w+f''(u)\Nx u\, w\|^2_{L^2}\le\\\le \|f'(u)\|_{L^\infty}^2\|\Nx w\|^2_{L^2}+\|f''(u)\|_{L^\infty}^2\|\Nx u\|^2_{L^\infty}\|w\|^2_{L^2}\le C \|\xi_w(t)\|^2_{\E},
\end{multline}
where the constant $C$ is independent of $t$, $\xi_0\in\mathcal A$ and $\hat\xi\in\E$. The Gronwall lemma applied to the differential inequality \eqref{2.var-est1} gives the desired estimate \eqref{2.var-est} and finishes the proof of the lemma.
\end{proof}
We are now ready to verify estimate \eqref{2.dif}. Indeed, let $\xi_0^1,\xi_0^2\in\mathcal A$ and let $u_1(t)$ and $u_2(t)$ be the corresponding Strichartz solutions of \eqref{eqCHO} with $\xi_{u_1}(0)=\xi_0^1$ and $\xi_{u_2}(0)=\xi_0^2$ respectively. Then, the difference $v(t):=u_2(t)-u_1(t)$ solves
\begin{equation}\label{2.dif12}
\Dt^2 v+\Dt v+\Dx(\Dx v-[f(u_2(t))-f(u_1(t))])+\alpha v=0,\ \ \xi_v\big|_{t=0}=\xi_0^2-\xi_0^1.
\end{equation}
Using that
\begin{equation}\label{2.dif-non}
f(u_2(t))-f(u_1(t))=\int_0^1f'(u_1(t)+s v(t))\,ds\, v(t)
\end{equation}
and arguing analogously to Lemma \ref{Lem2.var}, we have
\begin{equation}\label{2.var-est4}
\|\xi_v(t)\|_{\E}\le Ce^{Kt}\|\xi_v(0)\|_{\E},\ \ t\ge0,
\end{equation}
where the constant $C$ and $K$ are independent of $t$ and $\xi_0^i\in\mathcal A$.
\par
Let now $w(t)$ be a solution of the equations of variations \eqref{eq.var} where $u(t)$ is replaced by $u_1(t)$ and $\hat\xi:=\xi_0^2-\xi_0^1$. Then, obviously, $S'(t,\xi_0^1)(\xi_0^2-\xi_0^1)=\xi_w(t)$ and we need to estimate the energy norm of the difference $\theta(t):=v(t)-w(t)$. This difference solves the equation
\begin{equation}
\Dt^2\theta+\Dt\theta+\Dx\(\Dx\theta-[f(u_2)-f(u_1)-f'(u_1)w]\)+\alpha\theta=0,\ \ \xi_\theta\big|_{t=0}=0.
\end{equation}
Multiplying this equation by $\Dt(-\Dx)^{-1}\theta$, analogously to \eqref{2.var-est1}, we get
\begin{equation}\label{2.var-est3}
\frac12\frac d{dt}\|\xi_\theta(t)\|^2_{\E}\le \|\xi_\theta(t)\|^2_{\E}+\|\Nx(f(u_2)-f(u_1)-f'(u_1)w)\|_{L^2}^2.
\end{equation}
Using \eqref{2.dif-non}, we transform the last term on the RHS as follows
$$
f(u_2(t))-f(u_1(t))-f'(u_1(t))w(t)=\int_0^1[f'(u_1(t)+s v(t))-f'(u_1(t))]\,ds\, v(t)+f'(u_1(t))\theta(t).
$$
Thus, analogously to \eqref{2.good}, we have
\begin{equation}
\|\Nx(f(u_2)-f(u_1)-f'(u_1)w)\|_{L^2}^2\le C\|\xi_\theta(t)\|^2_{\E}+\int_0^1\|f'(u_1+s v)-f'(u_1)\|^2_{W^{1,\infty}}\,ds\,\|\xi_v(t)\|^2_{\E}.
\end{equation}
Since $f\in C^2(\R)$ and $u_1,u_2\in C^1_b(\R^3)$ it is not difficult to show that there exists a function $E(z)$ such that
$\lim_{z\to0}E(z)=0$ and
$$
\int_0^1\|f'(u_1(t)+s v(t))-f'(u_1(t))\|^2_{W^{1,\infty}}\,ds\le E(\|v(t)\|_{W^{1,\infty}})
$$
uniformly with respect to $t\ge0$ and $\xi_0^1,\xi_0^2\in\mathcal A$. Using now estimate \eqref{2.c1} together with the interpolation, we get
$$
\|v(t)\|_{W^{1,\infty}}\le C\|v(t)\|_{H^1}^\eta\|v(t)\|_{C^{1+\delta}}^{1-\eta}\le C\|\xi_v(t)\|_{\E}^\eta
$$
for the properly chosen  exponent $0<\eta<1$. Thus, using also \eqref{2.var-est4}, we end up with
$$
\|\Nx(f(u_2)-f(u_1)-f'(u_1)w)\|_{L^2}^2\le Ce^{Kt}\|\xi_v(0)\|^2_{\E}\, E\(Ce^{\eta Kt}\|\xi_v(0)\|_{\E}^\eta\)+C\|\xi_\theta(t)\|^2_{\E}
$$
uniformly with respect to $t\ge0$ and $\xi_0^1,\xi_0^2\in\mathcal A$. The Gronwall lemma applied to \eqref{2.var-est4} gives now the desired estimate \eqref{2.dif}. The continuity of the operator $\xi_0\to S'(t,\xi_0)$ can be established analogously and the theorem is proved.
\end{proof}
\begin{remark} We emphasize  that we have essentially used the extra regularity \eqref{2.c1} for the solutions belonging to the attractor in the proof given above, so this proof gives only quasidifferentiability of the operators $S(t)$ on the attractor and does not work for proving its Frechet differentiability in $\E$. Nevertheless, the Frechet differentiability of the solution operators $S(t)$ is  likely  true, but its proof is technically much more complicated (since instead of \eqref{2.c1} we have only the estimate in $\E$ in this case) and should also involve Strichartz-type estimates for the equations of variations \eqref{eq.var}. Since we need not Frechet differentiability for the proof of our main result, we will not give any more details here.
\end{remark}

\section{Proof of the main theorem}\label{s3}
In this section, we will prove our main result -- Theorem \ref{Th0.main} which establishes the finiteness of the fractal dimension of the global attractor $\mathcal A$ of the hyperbolic CHO equation in the energy space $\E$. To this end, we will use the volume contraction method and Theorem \ref{Th2.AVolContr}. According to this theorem, it is sufficient to verify that there exists time $T>0$ and $d\in\R_+$ such that
\begin{equation}\label{3.contr}
\omega_d(S'(T,\xi_0),\E)\le\frac12
\end{equation}
for all $\xi_0\in\mathcal A$. Indeed, the uniform quasidifferentiability of the map $S(T)$ is verified in Theorem \ref{Th2.var} and \eqref{3.contr} guarantees that $\omega_d(\mathcal A, S(T))<1$ and Theorem \ref{Th2.AVolContr} gives then that the fractal dimension of $\mathcal A$ does not exceed $d$. Moreover, equation of variations \eqref{eq.var} has the form of \eqref{1.lin-eq} in $\E$:
$$
\frac d{dt}\xi_w(t)=L(t,\xi_0)\xi_w(t),
$$
where
\begin{equation}
L(t,\xi_0):=\(\begin{matrix} 0 & 1\\-\alpha &-1\end{matrix}\)+\(\begin{matrix} 0&0\\-\Dx &0\end{matrix}\)\(\begin{matrix}\Dx-f'(u(t)) & 0\\0&0\end{matrix}\)
\end{equation}
and $\xi_u(t):=S(t)\xi_0$. Thus, the Liouville formula can be applied to verify assumption \eqref{3.contr}. According to
Theorem \ref{Th1.main}, we only need to find the equivalent norms $\|\cdot\|_{\E(t,\xi_0)}$ and the operators $\Cal C(t,\xi_0)$ and $\Cal K(t,\xi_0)$ which satisfy the assumptions of Theorem \ref{Th1.main} uniformly with respect to $\xi_0\in\mathcal A$.
In order to do so, we multiply equation \eqref{eq.var} by $(-\Dx)^{-1}(\Dt w+\delta v)$ where $\delta>0$ is a small parameter which will be specified below and integrate over $x\in\R^3$. Then, after the straightforward transformations, we will have at least formally
\begin{multline}\label{3.energy}
\frac12\frac d{dt}\(\|\xi_w\|^2_{\E}+2\delta(\Dt w,w)_{\dot H^{-1}}+\delta\|w\|^2_{\dot H^{-1}}\)=\\=-\((1-\delta)\|\Dt w\|^2_{\dot H^{-1}}+\delta\|w(t)\|^2_{\dot H^1}+\alpha\delta\|w(t)\|^2_{\dot H^{-1}}\)-(f'(u)w,\Dt w)-\delta(f'(u)w,w).
\end{multline}
The justification of this version of the energy equality can be done exactly as in \cite{SZ.CHO}. Formula \eqref{3.energy} prompts to take the quadratic form on the LHS of it as the desired equivalent metric $\|\cdot\|_{\E(t,\xi_0)}^2$. However, as not difficult to see, this will not work due to the presence of a "bad" term $(f'(u)w,\Dt w)$ on the RHS of this formula. Indeed, in contrast to the case of damped wave equation, say with Dirichlet boundary conditions, where natural energy space is $H^1_0(\Omega)\times L^2(\Omega)$, the term $(f'(u)w,\Dt w)$ can not be estimated via norms which are {\it compact} in the energy space $\E=H^1_0(\Omega)\times H^{-1}(\Omega)$ even if $f'$ is bounded and the equation is considered in a smooth bounded domain $\Omega$. In addition, the fact that the underlying domain $\Omega=\R^3$ is unbounded and the embedding $H^1(\Omega)\subset L^2(\Omega)$ is no more compact also requires an extra care. So, we need to proceed in a more delicate way involving the time dependent metrics and the properly chosen cut off functions. Namely, let $\psi_R=\psi_R(x)$ be the smooth cut-off function such that
$$
0\le\psi_R(x)\le1,\ \ \psi_R(x)=1 \text{ if } |x|\le R-1\ \text{and}\ \psi_R(x)=0\ \text{if}\ |x|\ge R,
$$
where $R\gg1$ is one more parameter which will be specified later. Let also
\begin{equation}\label{3.norms}
\|\xi_w(t)\|^2_{\E(t,\xi_0)}:=\|\xi_w\|^2_{\E}+2\delta(\Dt w,w)_{\dot H^{-1}}+\delta\|w\|^2_{\dot H^{-1}}+(f'(u)w,w)+L\|(-\Dx+1)^{-1/2}(\psi_Rw)\|^2_{L^2},
\end{equation}
where $L\gg1$ is one more parameter. Then, due to the energy equality \eqref{3.energy}, we have
\begin{multline}\label{3.energy1}
\frac12\frac d{dt}\|\xi_w(t)\|^2_{\E(t,\xi_0)}=-\((1-\delta)\|\Dt w\|^2_{\dot H^{-1}}+\delta\|w(t)\|^2_{\dot H^1}+\alpha\delta\|w(t)\|^2_{\dot H^{-1}}\)-\\-\delta(f'(u)w,w)+\frac{1}{2}(f''(u)\Dt u,w^2)+L((-\Dx+1)^{-1}(\psi_R w),\psi_R\Dt w):=(M(t,\xi_0)\xi_w,\xi_w)_{\E}.
\end{multline}
Now one can see that the "bad" term on the RHS is killed and the next lemma shows that the norms \eqref{3.norms} are equivalent to the standard norm of $\E$.

\begin{lemma}\label{Lem3.eq} Let the above assumptions hold. Then, there exist constants $\delta$, $R$ and $L$ such that
\begin{equation}\label{3.equiv}
c^{-1}\|\xi\|^2_{\E}\le\|\xi\|_{\E(t,\xi_0)}^2\le c\|\xi\|^2_{\E},\ \ \xi\in\E
\end{equation}
where the constant $c$ is independent of $\xi\in\E$, $t\in\R$ and $\xi_0\in\mathcal A$.
\end{lemma}
\begin{proof}Indeed, the right inequality is obvious since due to the control \eqref{2.c1}
$$
|(f'(u)w,w)|\le \|f'(u)\|_{L^\infty}\|w\|^2_{L^2}\le C\|\xi_w\|^2_{\E}.
$$
So, we only need to verify the left one. To this end, we note that assumption \eqref{f123}.1 on the nonlinearity $f$ implies that $f'(0)\ge0$ and, therefore, taking into the account \eqref{2.c1},
\begin{multline}
(f'(u)w,w)=(\psi_R f'(u)w,w)+((1-\psi_R)[f'(u)-f'(0)]w,w)+f'(0)((1-\psi_R)w,w)\ge\\\ge-C(\psi_R w,w) -C\|(1-\psi_R)u\|_{L^\infty}\|w\|^2_{L^2}\ge\\\ge -C\|(-\Dx+1)^{-1/2}(\psi_Rw)\|_{L^2}((-\Dx+1)w,w)^{1/2}-C\|(1-\psi_R)u\|_{L^\infty}\|w\|^2_{L^2}\ge\\\ge -C_\eb\|(-\Dx+1)^{-1/2}(\psi_Rw)\|_{L^2}^2-\eb\|\xi_w\|^2_{\E}-C\|(1-\psi_R)u\|_{L^\infty}\|w\|^2_{L^2},
\end{multline}
where $\eb>0$ is arbitrary.
Remind that according Theorem \ref{th.attr}, the attractor $\mathcal A$ is {\it compact} in $\E$ and is bounded in $\E_2$. Therefore, as not difficult to show using the interpolation inequality, for every $\eb>0$ there exists $R=R(\eb)$ such that
\begin{equation}\label{3.tail}
\|(1-\psi_R)u\|_{L^\infty}\le C\|u\|_{L^\infty(|x|>R-1)}\le\eb
\end{equation}
uniformly with respect to $\xi_u\in\mathcal A$.
Thus, for $R\ge R(\eb)$,
\begin{multline}\label{3.fpgood}
(f'(u)w,w)\ge -C\eb(\|w\|^2_{L^2}+\|\xi_w\|^2_{\E})-\\-C_\eb\|(-\Dx+1)^{-1/2}(\psi_Rw)\|^2_{L^2}\ge
-C\eb\|\xi_w\|^2_{\E}-C_\eb\|(-\Dx+1)^{-1/2}(\psi_Rw)\|^2_{L^2},
\end{multline}
where we have implicitly used that $\dot H^1\cap\dot H^{-1}\subset L^2$. Thus,
\begin{multline}\label{3.norms1}
\|\xi_w(t)\|^2_{\E(t,\xi_0)}\ge(1-C\eb)\|\xi_w\|^2_{\E}+2\delta(\Dt w,w)_{\dot H^{-1}}+\\+\delta\|w\|^2_{\dot H^{-1}}+(L-C_\eb)\|(-\Dx+1)^{-1/2}(\psi_Rw)\|^2_{L^2}.
\end{multline}
This estimate implies the desired right inequality of \eqref{3.equiv} if $\delta>0$ and $\delta>0$ are small enough, $L\ge C_\eb$ and $R\ge R(\eb)$ and finishes the proof of the lemma.
\end{proof}
Thus, to finish the proof of our main result, we only need to verify that the operator $M(t,\xi_0)$ defined in \eqref{3.energy1} satisfies assumptions \eqref{1.split}, \eqref{1.neg} and \eqref{1.comp} of Theorem \ref{Th1.main}. By elementary estimates we derive
\begin{multline}\label{3.est-add}
((-\Dx+1)^{-1}(\psi_R w),\psi_R\Dt w)\le \|\Nx(\psi_R(-\Dx+1)^{-1}(\psi_R w))\|_{L^2}\|\Dt w\|_{\dot H^{-1}}\le\\\le
C_R\|\psi_R w\|_{L^2}\|\xi_w\|_{\E}.
\end{multline}
Using now the fact that $u$ is uniformly bounded in $C_b(\R^3)$, $\Dt u$ is uniformly bounded in $H^1(\R^3)$ and the embedding $H^1(\R^3)\subset L^6(\R^3)$ together with the H\"older inequality, we get
\begin{multline}\label{3.est10}
|(f''(u)\Dt u,w^2)|\le C(|\Dt u|,\psi_Rw^2)+C((1-\psi_R)|\Dt u|,w^2)\le\\\le C\|\Dt u\|_{L^3}\|\psi_Rw\|_{L^2}\|w\|_{L^6}+C\|(1-\psi_R)\Dt u\|_{L^3}\|w\|_{L^6}\|w\|_{L^2}\le\\\le C\|\psi_R w\|_{L^2}\|\xi_w\|_{\E}+C\|(1-\psi_R)\Dt u\|_{L^3}\|\xi_w\|^2_{\E}.
\end{multline}
Since the global attractor $\mathcal A$ is compact in $\E$ and $\Dt u$ is uniformly bounded in $L^6$, analogously to \eqref{3.tail} the
interpolation inequality gives the following tail estimate: for every $\eb>0$ there exists $R=R(\eb)$ such that
\begin{equation}\label{3.tail1}
\|(1-\psi_R)\Dt u\|_{L^3}\le C\|\Dt u\|_{L^3(|x|>R-1)}\le\eb
\end{equation}
uniformly with respect to $\xi_u\in\mathcal A$.  Therefore, estimate \eqref{3.est10} for $R\ge R(\eb)$ reads
\begin{equation}\label{3.estfpp}
|(f''(u)\Dt u,w^2)|\le C\eb\|\xi_w\|^2_\E+C_{\eb}\|\psi_R w\|^2_{L^2}
\end{equation}
and combining estimates \eqref{3.est-add}, \eqref{3.estfpp} together with \eqref{3.fpgood}, we finally see that, for sufficiently small $\delta$ and $\eb$ and sufficiently large $R=R(\eb)$,
\begin{equation}\label{3.split}
(M(t,\xi_0)\xi_w,\xi_w)\le -\gamma\|\xi_w\|^2_{\E}+C_1\|\psi_R w\|^2_{L^2}=-\gamma\|\xi_w\|^2_{\E}+C_1(K\xi_w,\xi_w)_\E,
\end{equation}
where $\gamma>0$ and $C_1$ are independent of $\xi_w,\xi_0\in\E$ and $t\in\R$ and the operator $K$ is defined via
\begin{equation}\label{3.k}
(K\xi_w,\xi_w)_{\E}=\|\psi_R w\|^2_{L^2}
\end{equation}
(indeed, due to the Riesz representation theorem, $K$ is a bounded nonnegative self-adjoint operator in $\E$). Thus, \eqref{1.split} and \eqref{1.neg} are verified with $\mathcal C(t,\xi_0)=-\gamma Id$ and $\mathcal K(t,\xi_0):=C_1K$ and to finish the proof of the main result, we only need to verify that the operator $K$ is compact. This is done in the following lemma.

\begin{lemma} Let the above assumptions hold. Then the operator $K\in\mathcal L(\E,\E)$ defined via the quadratic form \eqref{3.k} is compact.
\end{lemma}
\begin{proof} Indeed, by the parallelogram law and the embedding $\dot H^1\cap\dot H^{-1}\subset L^2$.
$$
(K\xi,\bar\xi)_\E=(\psi_R\xi_1,\psi_R\bar\xi_1)\le \|\psi_R\xi_1\|_{L^2}\|\psi_R\bar\xi_1\|_{L^2}\le
C\|\psi_R\xi_1\|_{L^2}\|\bar\xi\|_{\E},
$$
where $\xi=(\xi_1,\xi_2)\in\E$ and $\bar\xi=(\bar\xi_1,\bar\xi_2)\in\E$. Therefore,
\begin{equation}\label{3.k-est}
\|K\xi\|_{\E}=\sup_{\bar\xi\in\E,\,\bar\xi\neq 0}\frac{(K\xi,\bar\xi)_{\E}}{\|\bar\xi\|_{\E}}\le C\|\psi_R\xi_1\|_{L^2}.
\end{equation}
Let now $\xi^n=(\xi^n_1,\xi^n_2)\in\E$ be a bounded sequence in $\E$. Then, due to the embedding $\dot H^1\cap\dot H^{-1}\subset H^1$,
the sequence $\xi_1^n$ is bounded in $H^1(\R^3)$. Since $\psi_R$ is smooth and has a finite support, the sequence $\psi_R\xi_1^n$ is bounded in $H^1(|x|<R)$ and, finally, since the embedding $H^1(|x|<R)\subset L^2(|x|<R)$ is compact, the sequence $\psi_R\xi_1^n$ is precompact in $L^2(\R^3)$. Thus, there exists a convergent in $L^2$ subsequence of $\psi_R\xi^n_1$ which we also denote by $\psi_R\xi_1^n$ for simplicity. Then, from \eqref{3.k-est}, we infer
$$
\|K(\xi^n-\xi^m)\|_{\E}\le C\|\psi_R\xi^n_1-\psi_R\xi_1^m\|_{L^2}
$$
and, therefore, $K\xi^n$ is a Cauchy sequence in $\E$. Since $\E$ is complete $K\xi^n$ is convergent and $K$ is compact. So, the lemma is proved.
\end{proof}
Thus, all of the assumptions of Theorem \ref{Th1.main} are verified and, consequently, estimate \eqref{3.contr} is proved and the main Theorem \ref{Th0.main} on the finite-dimensionality of the global attractor $\mathcal A$ in the energy phase space $\E$ is also proved.

\end{document}